\documentclass{amsart}
\usepackage{amsmath}
\usepackage{amsthm}
\usepackage{amsfonts,mathrsfs}
\usepackage{amssymb}
\usepackage{graphicx}
\usepackage{enumitem}
\usepackage{color}
\usepackage{verbatim}
\theoremstyle{plain}
\newtheorem{theorem}{Theorem}[section]
\newtheorem{lemma}[theorem]{Lemma}
\newtheorem{corollary}[theorem]{Corollary}
\newtheorem{proposition}[theorem]{Proposition}
\newtheorem{conjecture}[theorem]{Conjecture}
\theoremstyle{definition}
\newtheorem{definition}[theorem]{Definition}

\newtheoremstyle{TheoremNum}
	{\topsep}{\topsep}              
  {\itshape}                      
  {}                              
  {\bfseries}                     
  {.}                             
  { }                             
  {\thmname{#1}\thmnote{ \bfseries #3}}

\newcommand{\F}{\mathbb F}
\newcommand{\K}{\mathbb K}
\newcommand{\cK}{\mathcal K}
\newcommand{\U}{\mathcal U}
\newcommand{\Oval}{\mathcal O}

\newcommand{\sq}{\square_q}

\newcommand{\Tr}{ \ensuremath{ \mathrm{Tr}}}

\newcommand{\RN}[1]{%
  \textup{\uppercase\expandafter{\romannumeral#1}}%
}

 \def\zhou#1 {\fbox {\footnote {\ }}\ \footnotetext { From Yue: {\color{blue}#1}}}

 \def\trom#1 {\fbox {\footnote {\ }}\ \footnotetext { From Rocco: {\color{blue}#1}}}

\begin{document}
	\title{On the $2$-ranks of a class of unitals}
	\author[R. Trombetti]{Rocco Trombetti}
	\address{Dipartimento di Mathematica e Applicazioni ``R. Caccioppoli", Universit\`{a} degli Studi di Napoli ``Federico \RN{2}", I-80126 Napoli, Italy}
	\email{rtrombet@unina.it}
	\author[Y. Zhou]{Yue Zhou}
	\address{Dipartimento di Mathematica e Applicazioni ``R. Caccioppoli", Universit\`{a} degli Studi di Napoli ``Federico \RN{2}", I-80126 Napoli, Italy}
	\email{yue.zhou.ovgu@gmail.com}
	
	\begin{abstract}
		Let $\U_\theta$ be a unital defined in a shift plane of odd order $q^2$, which are constructed recently in \cite{trombetti_unitals_2015}. In particular, when the shift plane is desarguesian, $\U_\theta$ is a special Buekenhout-Metz unital formed by a union of ovals. We investigate the dimensions of the binary codes derived from $\U_\theta$.
		By using Kloosterman sums, we obtain a new lower bound on the aforementioned dimensions which	
		improves Leung and Xiang's result \cite{leung_dimensions_2009,leung_erratum_2011}. In particular, for $q=3^m$, this new lower bound equals $\frac{2}{3}(q^3+q^2-2q)-1$ for even $m$ and $\frac{2}{3}(q^3+q^2+q)-1$ for odd $m$.
	\end{abstract}
	\maketitle
\section{Introduction}
Let $m$ be an integer larger than or equal to $3$. A \emph{unital} of order $m$ is a $2$-$(m^3+1, m+1, 1)$ design, i.e.\ a set of $m^3+1$ points arranged into subsets of size $m+1$ such that each pair of distinct points are contained in exactly one of these subsets. 

Most of the known unitals can be embedded in a projective plane $\Pi$ of order $q^2$. In such a case, the \emph{embedded unital} is a set $\U$ of $q^3+1$ points such that each line of $\Pi$ intersects $\U$ in $1$ or $q+1$ points. When $\Pi$ is the desarguesian projective plane $\mathrm{PG}(2,q^2)$, the set of absolute points of a unitary polarity, or equivalently speaking, the rational points on a nondegenerate Hermitian curve form a \emph{classical} unital. There are also non-classical unitals in $\mathrm{PG}(2,q^2)$, for instance the Buekenhout-Metz unitals \cite{buekenhout_characterizations_1976}, as well as the unitals which can not be embedded in a projective plane, such as the Ree unitals \cite{luneburg_remarks_1966}. Moreover, it is not necessary that the order of a unital is a prime power, for instance, the order of the unitals discovered in \cite{bagchi_designs_1989} equals $6$.

Unitals also exist in non-desarguesian planes. For instance, there are unitals derived from unitary polarities in various translation planes and shift planes; see \cite{abatangelo_ovals_1999,abatangelo_polarity_2002,ganley_class_1972,ganley_polarities_1972,hui_non-classical_2013,knarr_polarities_2010}. Commutative semifield planes, as a special type of translation and shift  planes, also contain the unitals, which are analogous to the Buekenhout-Metz ones in desarguesian planes; see \cite{abatangelo_transitive_2001,zhou_parabolic_2015}. 

Recently in \cite{trombetti_unitals_2015}, the authors investigate the existence and properties of a special type of unitals $\U_\theta$ consisting of ovals  in shift planes $\Pi(f)$ of odd orders in terms of planar functions $f$ on $\F_{q^2}$. In particular, when the planar function $f(x)=x^2$, the shift plane $\Pi(f)$ is desarguesian and the unital $\U_\theta$ is exactly the one independently discovered by Hirschfeld and Sz\"{o}nyi \cite{hirschfeld_sets_1991} and by Baker and Ebert \cite{baker_intersection_1990}, which forms a special subclass of the Buekenhout-Metz unitals in desarguesian planes; see \cite{trombetti_unitals_2015} or Section \ref{sec:construction} for more details.

Generally, a \emph{linear code} is an arbitrary subspace of a vector space over a field. The \emph{dimension} of a linear code is the dimension of the corresponding subspace. 

Let $\U$ be a unital, namely, a $(q^3+1,q+1,1)$-design. For any prime number $p$, let $C_p(\U)$ be the subspace spanned by the characteristic vectors of the blocks of $\U$ in $\F_p^{q^3+1}$. Here the characteristic vector $v^B$ of a subset $B$ of the point set of $\U$, is the vector in $\F_p^{q^3+1}$ with coordinate $1$ in those positions corresponding to the elements in $B$ and with coordinate $0$ in all other positions.

The dimension of $C_p(\U)$ is also called the \emph{$p$-rank} of the design $\U$. It is worth noting that, as a design, $\U$ is of order $q^2-1$. By \cite[Theorem 2.4.1]{assmus_designs_1992},  $C_p(\U)$ is interesting only when $p\mid (q^2-1)$. In this paper, we consider in the value of $C_2(\U_\theta)$, where $\U_\theta$ is a unital in a shift plane $\Pi(f)$ of odd order constructed in \cite{trombetti_unitals_2015}; see Section \ref{sec:construction} too. As we mentioned previously, when $f(x)=x^2$, our unitals correspond to a special subclass of the Buekenhout-Metz unitals in desarguesian planes. Baker and Wantz made the following conjecture.
\begin{conjecture}\label{conjecture:2rank}
	When $f(x)=x^2$, $\dim C_2(\U_\theta)=q^3-q+1$.
\end{conjecture}

This conjecture can be found in \cite{barwick_unitals_2008,ebert_binary_codes_2001,xiang_recent_ams_2005}. In \cite{leung_dimensions_2009,leung_erratum_2011}, Leung and Xiang proved that $\dim C_2(\U_\theta) \ge (q^3-q^2+q) (1-\frac{1}{p}) + \frac{q^2}{p}$, where $p=\mathrm{char}(\F_q)$. The proof of this conjecture was claimed by Wu in a conference talk \cite{wu_binarycodes_2012} with few details. Nevertheless, the proof has not appeared in the public domain yet since 2012.

This paper is organized as follows: We first briefly introduce shift planes and the unitals $\U_\theta$ constructed in \cite{trombetti_unitals_2015}. Then we investigate the dimensions of the binary codes generated by the characteristic vectors of the blocks of the unitals $\U_\theta$. In particular, for $q=3^n$ and $f(x)=x^2$, we use Kloosterman sums to improve Leung and Xiang's result \cite{leung_dimensions_2009,leung_erratum_2011} on the aforementioned dimension.

\section{Shift planes and unitals}\label{sec:construction}

A projective plane is called a  \emph{shift plane} if there exists a flag $((\infty), L_\infty)$ and a commutative collineation group which fixes $((\infty), L_\infty)$ and acts regularly on the set of points not lying on $L_\infty$ as well as the set of lines not passing through $(\infty)$. A finite shift plane of order $q$ can be equivalently derived from abelian $(q,q,q,1)$-relative difference sets (RDS for short); see \cite{ganley_relative_1975}. 

When $q$ is odd, all known abelian $(q,q,q,1)$-RDSs are subsets of the group $(\F_q^2,+)$. Such a $(q,q,q,1)$-RDS is equivalent to a function $f:\F_q\rightarrow\F_q$, such that $x\mapsto f(x+a)-f(x)$ is always a bijection for each nonzero $a$. This type of functions are called \emph{planar functions} on $\F_q$, which were first investigated by Dembowski and Ostrom in \cite{dembowski_planes_1968}. As the counterpart, when $q=2^n$, abelian $(q,q,q,1)$-RDSs only exist in $C_{4}^n$ where $C_4$ is the cyclic group of order $4$. These RDSs can also be equivalently illustrated by functions over $\F_{2^n}$, which can be found in \cite{schmidt_planar_2014,zhou_2^n2^n2^n1-relative_2012}.

Let $\F$ be a finite field of an odd order and $f$ a planar function on $\F$. We define a projective plane $\Pi(f)$ as follows:
\begin{itemize}
	\item \textbf{Points:} $(x,y)\in \F\times\F$ and $(a)$ for $a\in \F\cup\{\infty\}$;
	\item \textbf{Lines:} $L_{a,b}:=\{(x,f(x+a)-b): x\in\F\}\cup \{(a)\}$ for all $(a,b)\in\F\times \F$, $N_a := \{(a,y): y\in\F \}\cup \{(\infty)\}$ and $L_\infty:= \{(a): a\in \F\cup \{\infty\} \}$.
\end{itemize}

The points except for those on $L_\infty$ are called the \emph{affine points} of $\Pi(f)$. By removing the line $L_\infty$ and the points on it, we get an affine plane.

The set of maps 
\[T:=\{\tau_{u,v}: \tau_{u,v}(x,y)=(x+u,y+v): u,v \in \F\}\]
induces an abelian collineation group on $\Pi(f)$, and this group acts regularly on the affine points and all lines $\{L_{a,b}: a,b\in \F\}$. Thus $\Pi(f)$ is a shift plane. We call this collineation group the \emph{shift group} of $\Pi(f)$. 

When $f$ can be written as a Dembowski-Ostrom polynomial, i.e.\ $f(x)=\sum a_{ij}x^{p^i+p^j}$ where $p=\mathrm{char}(\F)$, the plane $\Pi(f)$ is also a commutative semifield plane. Using the corresponding semifield multiplication, we can label the points and lines of $\Pi(f)$ in a different way. The intersection of the translation group and the shift group of $\Pi(f)$ is $\{(x,y)\mapsto (x,y+b): b\in \F\}$. See \cite[Section 4]{ghinelli_finite_2003} for details. We refer to \cite{lavrauw_semifields_2011} and \cite{pott_semifields_2014} for recent surveys on semifields and planar functions respectively.

Up to equivalence, all known planar functions $f$ on finite fields $\F_q$ of odd characteristics can be written as a Dembowski-Ostrom polynomial except for the Coulter-Matthews ones which are power maps defined by $x\mapsto x^d$ on $\F_{3^m}$ for certain $d$; see \cite{coulter_planar_1997}. Both the Dembowski-Ostrom planar functions and the Coulter-Matthews ones satisfy that
\begin{itemize}
	\item $f(0)=0$ and 
	\item for arbitrary $a,b\in \F_q$, $f(a)=f(b)$ if and only if $a=\pm b$.
\end{itemize}
For a proof of the Dembowski-Ostrom polynomials case, we refer to \cite{kyureghyan_theorems_2008}; for the Coulter-Matthews functions $f(x)=x^d$ on $\F_{3^m}$, it can be verified directly from the fact $\gcd(d,3^m-1)=2$. Actually for a function $f$ defined by a Dembowski-Ostrom polynomial, the above conditions are necessary and sufficient for $f$ to be planar; see \cite{weng_further_2010}. If a planar function satisfies the aforementioned two conditions, then we call it a \emph{normal planar function}.

Let $\xi$ be an element in $\F_{q^2}\setminus \F_q$. Then every element $x$ of $\F_{q^2}$ can be written as $x=x_0+x_1\xi$ where $x_0,x_1\in\F_q$. Similarly, every function $f:\F_{q^2}\rightarrow \F_{q^2}$ can be written as  $f(x)=f_0(x)+f_1(x)\xi$ where $f_0,f_1$ are maps from $\F_q$ to itself. Throughout this paper, we frequently switch between the element $x\in \F_{q^2}$ and its two dimensional representation $(x_0,x_1)\in \F_{q}^2$. If a special assumption on $\xi$ is needed, we will point it out explicitly.

In \cite{trombetti_unitals_2015}, it is proved that the set of points
\begin{equation}\label{eq:u_theta_general}
	\U_\theta:=\{(x,t\theta): x\in\F_{q^2}, t\in \F_q \} \cup \{(\infty)\} 
\end{equation}
is a unital in $\Pi(f)$ under the assumption that 
\[\#\{x\in\F_{q^2}:\theta_1 f_0(x) - \theta_0 f_1(x)=c\} =
	\left\{
	\begin{array}{ll}
	 q+1, & \hbox{$c\neq 0$;} \\
	 1, & \hbox{$c=0$.}
	\end{array}
\right.
\]
By choosing appropriate elements $\theta$, it is shown that $\U_\theta$ are unitals for $8$ distinct families of planar functions $f$; see \cite{trombetti_unitals_2015}.

As a design, the point set of $\U_\theta$ is $\{(x,t\theta): x \in \F_{q^2}, t\in \F_q\}\cup \{(\infty)\}$ and all of its blocks are
\[B_{a}:= \{(a,t\theta): t\in \F_q\}\cup \{(\infty) \},\]
for each $a\in \F_{q^2}$ and
\[B_{a,b}:=\{(x,t\theta): f(x+a)-b=t\theta, t\in \F_q \}, \]
for each $a,b\in \F_{q^2}$ where $b_0\theta_1-b_1\theta_0\neq 0$. The equation $f(x+a)-b=t\theta$ is equivalent to the following two ones
\begin{eqnarray*}
	 f_0(x) - b_0 &=& t\theta_0,\\
	 f_1(x) - b_1 &=& t\theta_1.
\end{eqnarray*}
As $\theta\neq0$, the above system of equations is equivalent to
\begin{eqnarray*}
	\theta_1f_0(x)-\theta_0f_1(x)-(\theta_1b_0-\theta_0b_1) &=& 0,\\
				f_1(x) - b_1 &=& t\theta_1.
\end{eqnarray*} 
(If $\theta_1=0$, then we replace the second equation by $f_0(x)-b_0=t\theta_0$.)
Hence, the block $B_{a,b}$ can also be written as
\begin{equation}\label{eq:B_ab_another}
	B_{a,b}=\left\{\left(x,\frac{f_1(x+a)-b_1}{\theta_1}\cdot \theta\right): f_0(x+a) \theta_1-f_1(x+a)\theta_0=b_0\theta_1-b_1\theta_0 \right\},
\end{equation}
where $b_0\theta_1-b_1\theta_0\neq 0$ and $\theta_1\neq 0$. For $\theta_1=0$, we get
\begin{equation}\label{eq:B_ab_another_0}
	B_{a,b}=\left\{\left(x,\frac{f_0(x+a)-b_0}{\theta_0}\cdot \theta\right): f_0(x+a) \theta_1-f_1(x+a)\theta_0=b_0\theta_1-b_1\theta_0 \right\}.
\end{equation}

There are totally $q^4-q^3+q^2$ blocks and each of them contains $q+1$ points. For each pair of points, there is exactly one block containing them both. Hence $\U_\theta$ is a $2$-$(q^3+1,q+1,1)$-design.

It is not difficult to directly verify the following property of $\U_\theta$.

\begin{proposition}\label{prop:U_is_transitive}
	Let $\U_\theta$ be a unital defined by \eqref{eq:u_theta_general}. The subgroup $T_\theta := \{\tau_{a,b\theta}: a\in \F_{q^2}, b\in \F_q\}$ of the shift group $T$ of $\Pi(f)$ acts regularly on the affine points of $\U_\theta$.
\end{proposition}

An \emph{oval} $\Oval$ in a projective plane $\Pi$ of odd order $q$ is a set of $q+1$ points such that every line in $\Pi$ meets $\Oval$ in $0$, $1$ or $2$ points. According to the famous result by Segre in \cite{segre_ovals_1955}, all ovals in desarguesian planes of odd orders are nondegenerate conics. The following property of $\U_\theta$ is proved in \cite{trombetti_unitals_2015}:
\begin{proposition}\label{prop:unital=ovals}
	Let $f$ be a normal planar function on $\F_{q^2}$ and $\U_\theta$ be a unital in $\Pi(f)$ defined by \eqref{eq:u_theta_general}. Then for each $c\in \F_{q^2}$, the set 
	\[\Oval_c:=\{(x, c): x \in \F_{q^2}\}\cup \{(\infty)\}\]
	is an oval in $\Pi(f)$ and $\U_\theta$ is a union of ovals, i.e.\
	\[\U_\theta = \bigcup_{t\in\F_q} \Oval_{t\theta}.\] 
\end{proposition}

In particular, when $f(x)=x^2$ on $\F_{q^2}$, $\Pi(x^2)$ is a desarguesian plane and the following results hold:
\begin{lemma}\label{le:known_x^2}
	Let $\theta$ be in $\F_{q^2}$ such that $\theta^{q+1}$ is a nonsquare element in $\F_q$. 
	\begin{enumerate}
		\item The set of points $\U_\theta$ defined by \eqref{eq:u_theta_general} is a unital of order $q$ in the plane $\Pi(x^2)$ \cite[Theorem 2.4]{trombetti_unitals_2015}.
		\item All the unitals in $\{\U_\theta: \theta^{q+1}\text{ is a nonsquare in }\F_q\}$ are equivalent under the collineations of $\Pi(x^2)$ \cite[Proposition 3.3]{trombetti_unitals_2015}.
		\item The unital $\U_\theta$ is exactly the one constructed by Hirschfeld and Sz\"{o}nyi \cite{hirschfeld_sets_1991} and by Baker and Ebert \cite{baker_intersection_1990} independently \cite[Remark 1]{trombetti_unitals_2015}.
	\end{enumerate} 
\end{lemma}

\section{Binary codes of $\U_\theta$ and their dimensions}\label{sec:codes}
By using MAGMA \cite{Magma} programs, we verified that for every planar functions on $\F_{q^2}$ constructed in \cite{trombetti_unitals_2015} with $q\le 9$, the $2$-ranks of all unitals $\U_\theta$ equal to $q^3-q+1$. Hence it seems that Conjecture \ref{conjecture:2rank} should also hold for other $\U_\theta$, i.e.\ $\dim C_2(\U_\theta)=q^3-q+1$ holds for all unitals $\U_\theta$ defined by \eqref{eq:u_theta_general}. 

The following upper bound on the dimension of $C_2(\U_\theta)$ was first proved for $f(x)=x^2$ (See \cite{leung_dimensions_2009,leung_erratum_2011} and \cite[Theorem 6.23]{barwick_unitals_2008}), and it can be generalized to all the unitals $\U_\theta$ constructed in \cite{trombetti_unitals_2015}. 
\begin{proposition}\label{prop:upperbound}
	Let $f$ be a normal planar function on $\F_{q^2}$ and $\U_\theta$ a unital defined by \eqref{eq:u_theta_general}.
	The characteristic vectors $v^{\Oval_{t\theta}}$ for all $t\in \F_q$, are linearly independent in $C_2(\U_\theta)^{\perp}$. Moreover, $\dim C_2(\U_\theta)\le q^3-q+1$.
\end{proposition}
\begin{proof}
	A vector $w$ lies in the dual of the binary code $C_2(\U_\theta)$ if and only if each block of $\U_\theta$ meets the support of $w$ in an even number of points. Here the support of $w$ is the set of points which correspond to the positions at which $w$ is nonzero. For every block $B_a$ of $\U_\theta$ containing $(\infty)$, it meets $\Oval_{t\theta}$ in exactly two points. For every block $B_{a,b}$ of $\U_\theta$, it meets $\Oval_{t\theta}$ in $0$ or $2$ points. Since $\U_\theta = \bigcup_{t\in\F_q} \Oval_{t\theta}$, all $v^{\Oval_{t\theta}}\in C_2(\U_\theta)^{\perp}$. As all these ovals have only the point $(\infty)$ in common, $\{v^{\Oval_{t\theta}}:t\in \F_q\}$ are linearly independent. It implies that $\dim (C_2(\U_\theta)^{\perp})\ge q$. Therefore $\dim C_2(\U_\theta)\le q^3-q+1$.
\end{proof}

Instead of considering $\dim C_2(\U_\theta)$, we remove the point $(\infty)$ from the point set of $\U_\theta$ and each block $B_a$. The new incidence structure is denoted by $\U'_\theta$. In other words,  $C_2(\U'_\theta)$ is the code $C_2(\U_\theta)$ punctured the coordinate corresponding to $(\infty)$. From the proof of Proposition \ref{prop:upperbound}, we know that all $v^{\Oval_{t\theta}}\in C_2(\U_\theta)^{\perp}$, which implies that $v^{\{(\infty)\}}\notin C_2(\U_\theta)$. Hence 
\begin{equation}\label{eq:dim_U=dim_U'}
\dim C_2(\U_\theta)=\dim C_2(\U'_\theta).
\end{equation}

Now we proceed to use the group characters approach applied in \cite{leung_dimensions_2009} to calculate $\dim C_2(\U'_\theta)$.
By Proposition \ref{prop:U_is_transitive}, the group $T_\theta$ acts transitively on all points of $\U_\theta$ except for $(\infty)$. Hence we can identify each coordinate of $C_2(\U'_\theta)$ with the elements in $T_\theta$. That means the point $(x,t\theta)\in \U_\theta$ corresponds to $(x,t\theta)\in T_\theta$. Under this identification, the code $C_2(\U'_\theta)$ becomes an ideal of the group ring $\F_2[T_\theta]$ and we can use the characters on $G$ to calculate $\dim C_2(\U'_\theta)$.

First we have to extend $\F_2$. Let $\K$ be a finite extension of $\F_2$ such that a primitive $p$-th root of unity $\varepsilon_p$ is in $\K$. It is well known that
\[
	\dim C_2(\U'_\theta)=\dim_{\K} C_2(\U'_\theta)=\#\{\chi \in \hat{T}_\theta: M e_\chi\neq 0  \},
\]
where $M$ is the submodule generated by the blocks of the incidence structure $\U'_\theta$,  $\hat{T}_\theta$ is the character group of $T_\theta$ and 
$e_\chi = \frac{1}{|T_\theta|} \sum_{g\in T_\theta} \chi(g^{-1})g$; see \cite[Page 277]{lander_symmetric_1983}.
Here $M e_\chi\neq 0$ means that $M e_\chi$ is not the trivial submodule $\{0\}$, which implies  there exists at least one element $B$ in $M$ such that $B e_\chi\neq 0$.
In other words,
\begin{equation}\label{eq:dim=cK}
	\dim C_2(\U'_\theta)= \#\cK(\U_\theta),
\end{equation}
where $\cK(\U_\theta)$ is defined to be the set of $\chi \in \hat{T}_\theta$ such that there exists at least one block $B$ of $\U'_\theta$ satisfying $\chi(B)\neq 0$. 

As $T_\theta\cong (\F_q^3,+)$, each character $\chi \in \hat{T}_\theta$ can be written as
\[\chi_{u,v,w}: (x,t\theta)\mapsto \varepsilon_p^{\Tr_{q/p}(ux_0+vx_1 + wt)},\]
where $u,v,w,t\in \F_q$ and $x=(x_0,x_1)\in \F_q^2$. For $t\in \F_q$, we also define
\[\chi(t)= \varepsilon_p^{\Tr_{q/p}(t)}.\]
That means $\chi_{u,v,w}(x,t\theta)=\chi(ux_0+vx_1+w\theta)$. The well known \emph{orthogonal relation} is 
\begin{equation}\label{eq:orthogonal}
	\sum_{x\in \F_q} \chi(ax)=
	\left\{
	  \begin{array}{ll}
	    1, & \hbox{$a= 0$;} \\
	    0, & \hbox{$a\neq 0$.}
	  \end{array}
	\right.
\end{equation}

\begin{lemma}\label{le:2rank-1}
	Let $f$ be a normal planar function on $\F_{q^2}$. Let $\U_\theta$ be a unital defined by \eqref{eq:u_theta_general}. Then $\chi_{u,v,0}\in\cK(\U_\theta)$ and $\chi_{0,0,w}\notin \cK(\U_\theta)$ for each $u,v\in \F_q$ and $w\in \F_q^*$. 
\end{lemma}
\begin{proof}
	Recall that we write $x=x_0+x_1\xi$ for each $x\in \F_{q^2}$.
	Without loss of generality, we assume that $\theta_1\neq0$ (otherwise $\theta_0\neq 0$, we replace $\theta_1$ and $f_1$ by $\theta_0$ and $f_0$ respectively in the rest of this proof). By \eqref{eq:B_ab_another}, we have that
	\[B_{a,b}=\left\{\left(x,\frac{f_1(x+a)-b_1}{\theta_1}\cdot \theta\right): f_0(x+a) \theta_1-f_1(x+a)\theta_0=b_0\theta_1-b_1\theta_0 \right\},\]
	where $b_0\theta_1-b_1\theta_0\neq 0$.  Let $w'= w/\theta_1$.
	We denote 
	\begin{equation}\label{eq:definition_circles}
		C_{a,\beta(b)}:=\{x : f_0(x+a)\theta_1-f_1(x+a)\theta_0=\beta(b)\},
	\end{equation}
	where $\beta(b)=b_0\theta_1-b_1\theta_0\neq 0$. It follows that 
	\begin{align*}
		\chi_{u,v,w}(B_{a,b}) &=\sum_{x\in C_{a,\beta(b)}}\chi\left(ux_0+vx_1+w\cdot \frac{f_1(x+a)-b_1}{\theta_1}\right)\\
						      &=\chi(-w'b_1)\sum_{x\in C_{0,\beta(b)}}\chi(u(x_0-a_0)+v(x_1-a_1)+w'f_1(x))\\
						      &=\chi(-ua_0-va_1-w'b_1)\sum_{x\in C_{0,\beta(b)}}\chi(ux_0+vx_1+w'f_1(x)),\\
		\chi_{u,v,w}(B_{a}) &=\sum_{t\in \F_{q}}\chi(ua_0+va_1+wt).				      
	\end{align*}
	
	When $u=v=0$ and $w\neq0$, by \eqref{eq:orthogonal}, we have 
	\[\chi_{0,0,w}(B_{a})=\sum_{t\in \F_{q}}\chi(wt)=0.\] 
	For each $c\in C_{0,\beta(b)}$, from the normality of $f$ it follows that $c\neq 0 $, $-c\in C_{0,\beta(b)}$ and $f_1(c)=f_1(-c)$. Hence 
	\[\chi_{0,0,w}(B_{a,b})=\chi(-w'b_1)\sum_{x\in C_{0,\beta(b)}}\chi(w'f_1(x))=0.\]
	That means $\chi_{0,0,w}\notin \cK(\U_\theta)$.
	
	When $w=0$, we have
	\[\chi_{u,v,0}(B_{a})=\sum_{t\in \F_{q}}\chi(ua_0+va_1)=q\chi(ua_0+va_1)=\chi(ua_0+va_1).\]
	Therefore $\chi_{u,v,0}\in \cK(\U_\theta)$.
\end{proof}

In the rest of this paper, we restrict ourselves to several special cases of $f$. The following theorem applies to the planar functions $f(x)=x^2$,  the one derived from Dickson's semifields and the semifields constructed in \cite{zhou_new_2013}. In fact, for $f(x)=x^2$, this result is proved by Leung and Xiang in \cite{leung_dimensions_2009,leung_erratum_2011} in a slightly different way.
\begin{theorem}\label{th:2rank_xiang}
	Let $f$ be a normal planar function on $\F_{q^2}$. Let $\U_\theta$ be a unital defined by \eqref{eq:u_theta_general}. Assume that $\theta_1\neq 0$ and $f_1(x)= 2x_0x_1$. If $w\neq 0$ and $\Tr_{q/p}\left(\frac{uv}{w}\theta_1\right)\neq 0$, then $\chi_{u,v,w}\in\cK(\U_\theta)$. Furthermore $\dim C_2(\U_\theta)\ge (q^3-q^2+q)\left(1-\frac{1}{p}\right) + \frac{q^2}{p}$.
\end{theorem}
\begin{proof}
	From the proof of Lemma \ref{le:2rank-1}, we have $w'=w/\theta$ and
	\[\chi_{u,v,w/2}(B_{a,b})=\chi(-ua_0-va_1-w'b_1/2)\sum_{x\in C_{0,\beta(b)}}\chi(ux_0+vx_1+w'f_1(x)/2).\]
	Hence we may only concentrate on the blocks $B_{0,b}$ and we denote $\sum_{x\in C_{0,\beta(b)}}\chi(ux_0+vx_1+w'f_1(x)/2)$ by $S(\beta)$. We then have
	\begin{align*}
		\sum_{\beta\in \F_q^*}S(\beta) &= \sum_{\beta\in \F_q^*} \sum_{x\in C_{0,\beta}}\chi(ux_0+vx_1+w'x_0x_1)\\
									   &= \sum_{x\in \F_{q^2}^*}\chi(ux_0+vx_1+w'x_0x_1)\\
									   &= \sum_{x_0\in \F_q}\chi(ux_0)\sum_{x_1\in \F_q}\chi((v+w'x_0)x_1) - \chi(0)\\
									   &= \chi\left(\frac{uv}{w'}\right)-\chi(0) \qquad \qquad \qquad \text{(by \eqref{eq:orthogonal})}\\
									   &= \chi\left(\frac{uv}{w'}\right)+1,
	\end{align*}
	which equals $0$ if and only if $\Tr_{q/p}\left(\frac{uv}{w'}\right)=0$. If $\sum_{\beta\in \F_q^*}S(\beta)\neq0$, then there is at least one nonzero term which means $\chi_{u,v,w}\in\cK(\U_\theta)$.
	
	To get the lower bound for $\cK(\U_\theta)$, we need to count the cardinality of the set $\{(u,v,w): w\neq 0, \Tr_{q/p}\left(\frac{uv}{w}\theta_1\right)=0\}$, which equals $(q-1)^2(1+\frac{q}{p})$. Together with \eqref{eq:dim_U=dim_U'}, \eqref{eq:dim=cK} and Lemma \ref{le:2rank-1}, we have
	\begin{align*}
		\dim C_2(\U_\theta) & = \#\cK(\U_\theta)\\
		& \ge q^2 + \left((q-1)(q^2-1) - (q-1)^2\left(1+\frac{q}{p}\right)\right)\\
		&=(q^3-q^2+q)\left(1-\frac{1}{p}\right) + \frac{q^2}{p}. \qedhere
	\end{align*}	
\end{proof}

\section{A new lower bound on the dimension of $C_2(\U_\theta)$}
In this section, we proceed to improve the lower bound on $\cK(\U_\theta)$ for $f(x)=x^2$. In the proof of our theorem, we need the following lemmas.
\begin{lemma}\cite[Theorem 6.26]{lidl_finite_1997}\label{le:quadratic_form_even}
	Let $f$ be a nondegenerate quadratic form over $\F_q$, $q$ odd, in an even number $n$ of indeterminates. Then for $b\in \F_q$ the number of solutions of the equation $f(x_1,\dots, x_n)=b$ in $\F_{q^n}$ is
	\[q^{n-1} + v(b)q^{(n-2)/2}\eta((-1)^{(n/2}\Delta),\]
	where $\eta$ is the quadratic character of $\F_q$,  $\Delta=\det(f)$ and the integer-valued function $v$ is defined by $v(b)=-1$ for $b\in \F_{q}^*$ and $v(0)=q-1$.
\end{lemma}

\begin{lemma}\cite[Theorem 6.27]{lidl_finite_1997}\label{le:quadratic_form_odd}
	Let $f$ be a nondegenerate quadratic form over $\F_q$, $q$ odd, in an odd number $n$ of indeterminates. Then for $b\in \F_q$ the number of solutions of the equation $f(x_1,\dots, x_n)=b$ in $\F_{q^n}$ is
	\[q^{n-1} + q^{(n-1)/2}\eta((-1)^{(n-1)/2}b\Delta),\]
	where $\eta$ is the quadratic character of $\F_q$ and $\Delta=\det(f)$.
\end{lemma}

\begin{lemma}\label{le:chi(ac^2)}
	Let $a$ be a nonzero element in $\F_q$. Then
	\[\sum_{c\in \F_q}\chi(ac^2)=1.\]
\end{lemma}
\begin{proof}
	As $(x,y) \mapsto \Tr_{q/p}(a(x+y)^2)-\Tr_{q/p}(ax^2)-\Tr_{q/p}(ay^2)$ defines a nondegenerate bilinear form on $\F_{p}^n$, the function $x\mapsto \Tr_{q/p}(ax^2)$ defines a nondegenerate quadratic form over $\F_p$. By Lemmas \ref{le:quadratic_form_even} and \ref{le:quadratic_form_odd}, we see that 
	\[\#\{x: \Tr_{q/p}(ax^2)=b\}\equiv\left\{
	  \begin{array}{ll}
	    0 \pmod{2}, & \hbox{$b\neq0$;} \\
	    1 \pmod{2}, & \hbox{$b=0$.}
	  \end{array}
	\right.
	\]
	Therefore $\sum_{c\in \F_q}\chi(ac^2)=\varepsilon_p^0=1$.
\end{proof}

Let $\zeta_p$ be a primitive $p$-th root of unity in $\mathbb{C}$ and $\mathcal{O}_p$ the algebraic integer ring in $\mathbb{Q}(\zeta_p)$. The character function $\lambda:\F_q\rightarrow\mathbb{Q}[\zeta_p]$ is defined by $\lambda(x)=\zeta_p^{\Tr_{q/p}(x)}$. As $\gcd(2,p)=1$, it is well known that the ideal generated by $2$ in $\mathcal{O}_p$ can be uniquely factorized into prime ideals
\[(2)=\mathfrak{p}_1\mathfrak{p}_2\dots \mathfrak{p}_g,\]
for certain $g$; see \cite[Chapter 13]{ireland_classical_1990}. The following lemma is just a direct observation.
\begin{lemma}\label{le:liftingtoQ}
	Let $A$ be a finite multiset whose elements are in $\F_q$. In $\K$, the character sum $\sum_{c\in A} \chi(c) =0$
	if and only if there is a prime ideal $\mathfrak{p}_i$ containing the ideal generated by $2$ such that
	$\sum_{c\in A} \lambda(c)\in \mathfrak{p}_i$. In other words, $\sum_{c\in A} \chi(c) \neq0$ if and only if $\sum_{c\in A} \lambda(c)\not\equiv 0\pmod{2}$.
\end{lemma}

\begin{definition}
	For $a\in \F_q$, we define the \emph{Kloostman sum} over $\F_q$ by
	\[K(a):=\sum_{x\in \F_q^*} \lambda(x^{-1}+ax).\]
\end{definition}
It is worthy noting that in general the value of a Kloosterman sum $K(a)$ is always real, because the complex conjugate of it equals
\[\bar{K}(a)=\sum_{x\in \F_q^*} \lambda(-x^{-1}-ax)=\sum_{x\in \F_q^*} \lambda((-x)^{-1}+a(-x))=K(a).\]

Next we proceed with our main result in this section. We take $f(x)=x^2$ for $x\in \F_{q^2}$. Let $\omega$ be a primitive element of $\F_{q^2}$. Depending on the value of $q$, we can choose $\theta$ in the following two ways such that $\theta^{q+1}$ is a nonsquare in $\F_q$.
\begin{itemize}
	\item When $-1$ is a square in $\F_q$, i.e.\ $q\equiv 1\pmod{4}$, we take $\xi=\omega^{(q+1)/2}$ and $\theta=\xi$. It is readily verified that $\xi^2$ and $\theta^{q+1}$ are both nonsquares in $\F_q$. 
	\item When $-1$ is not a square in $\F_q$, i.e.\ $q\equiv 3\pmod{4}$, we let $\xi\in\F_{q^2}\setminus \F_q$ such that $\xi^q=-\xi$. By Lemma \ref{le:quadratic_form_even} with $n=2$, it is not difficult to show that there always exists $\theta_0\in \F_q^*$ such that $\theta_0^2-\alpha$ is a nonsquare. We take $\theta = \theta_0 + \xi$. It is clear that $\theta^{q+1}=\theta_0^2-\alpha$ is a nonsquare.
\end{itemize}
By Lemma \ref{le:known_x^2}, we know that $\U_\theta$ is a unital in each of these two cases.  
Since all unitals $\U_\theta$ in $\Pi(x^2)$ are equivalent under certain collineations (see Lemma \ref{le:known_x^2}), we only have to handle with $\U_\theta$ where $\theta$ takes the aforementioned special values.
In next theorem, we show a link between $S(\beta)$ and certain Kloosterman sums.

\begin{theorem}\label{th:2rank-Kloostermansum}
	Let $u,v,w\in \F_q$ satisfying that $w\neq 0$ and $\Tr_{q/p}(uv/w^2)=0$. 
	Let $\U_\theta$ be a unital embedded in $\Pi(x^2)$ defined in Lemma \ref{le:known_x^2}.
	\begin{enumerate}[label=(\alph*)]
		\item Assume that $q\equiv 1 \pmod{4}$. Let $\theta$ and $\xi$ be both equal to $\omega^{(q+1)/2}$, where $\omega$ is a primitive element of $\F_{q^2}$. Let $\alpha=\xi^2$. Then $\chi_{u,v,w}\in \cK(\U_\theta)$ if one of the following collections of conditions are satisfied:
		\begin{itemize}
			\item $v=0$, $u\neq0$  and  $K\left(-\frac{u^4}{64w^2}\cdot \alpha \right)\not\equiv 2 \pmod{4}$;
			\item $u=0$, $v\neq 0$ and $K\left(-\frac{v^4}{64w^2}\cdot \frac{1}{\alpha}\right)\not\equiv 2 \pmod{4}$.
		\end{itemize}
		\item Assume that $q\equiv 3 \pmod{4}$. Let $\xi\in\F_{q^2}\setminus \F_q$ be such that $\xi^q=-\xi$ and $\theta=\theta_0+\xi$, where $\theta_0$ is such that $\theta_0^2-\alpha$ is a nonsquare in $\F_q$. Then $\chi_{u,v,w}\in \cK(\U_\theta)$ if one of the following collections of conditions are satisfied:
		\begin{itemize}
			\item $v=0$, $u\neq0$  and  $K\left(\frac{u^4}{64w^2}\cdot (\theta_0^2-\alpha) \right)\not\equiv 0 \pmod{4}$;
			\item $u=0$, $v\neq 0$ and $K\left(\frac{v^4\alpha^2}{64w^2}\cdot (\theta_0^2-\alpha) \right)\not\equiv 0 \pmod{4}$.
		\end{itemize} 
	\end{enumerate}
\end{theorem}
\begin{proof}
	Let us first look at the expressions of $C_{0,\beta}$ (see \eqref{eq:definition_circles}).
	\begin{itemize}
		\item When $q\equiv 1\pmod{4}$, we have $\theta_0=0$ and $\theta_1=1$. Hence
		\begin{equation}\label{eq:C_0beta-1}
			C_{0,\beta}=\{(x_0,x_1): x_0^2+\alpha x_1^2=\beta\}.	
		\end{equation}
		\item When $q\equiv 3\pmod{4}$, we have
		\begin{equation}\label{eq:C_0beta-2}
			C_{0,\beta}=\{(x_0,x_1): x_0^2-2\theta_0x_0x_1+\alpha x_1^2=\beta\}.
		\end{equation}
	\end{itemize}
	
	As we take $f(x)=x^2$, $x\in C_{0,\beta}$ if and only if $-x\in C_{0,\beta}$. Hence we may define $C_{0,\beta}^{(+)}$ as a complete set of coset representatives of the subgroup $\{1,-1\}$ in $C_{0,\beta}$ and $C_{0,\beta}^{(-)}:=C_{0,\beta}\setminus C_{0,\beta}^{(+)}$. Similarly for $\F_q^*$, we define $\F_{q}^{(+)}$ and $\F_{q}^{(-)}$. For convenience, we use $\sq$ to denote the square elements in $\F_q^*$.
	
	Our idea is to investigate $\sum_{c\in \sq}S(c)$ and $\sum_{c\in \F_q^*\setminus\sq}S(c)$ respectively. If we can show that one of them is not zero, then there exists at least one $S(c)\neq0$, which means that $\chi_{u,v,w}\in \cK(\U_\theta)$. 
	
	First let us calculate $\sum_{c\in \sq}S(c)$.
	\begin{align*}
		\sum_{c\in \sq}S(c) &= \sum_{d\in \F_q^{(+)}}S(d^2)\\
							&= \sum_{d\in \F_q^{(+)}} \sum_{x\in C_{0,d^2}}\chi(ux_0+vx_1+wx_0x_1).
	\end{align*}
	As the points in \eqref{eq:C_0beta-1} and \eqref{eq:C_0beta-2} are both defined by nondegenerate quadratic forms, it implies that,
	\begin{align*}
		\sum_{c\in \sq}S(c) &= \sum_{d\in \F_q^{(+)}} \sum_{x\in C_{0,1}}\chi(udx_0+vdx_1+wd^2x_0x_1)\\
							&= \sum_{x\in C_{0,1}^{(+)}}\sum_{d\in \F_q^{(+)}} (\chi(udx_0+vdx_1+wd^2x_0x_1)\\
							&\phantom{= \sum_{x}\sum_{d}} + \chi(-udx_0-vdx_1+wd^2x_0x_1))\\
							&=\sum_{x\in C_{0,1}^{(+)}}\sum_{d\in \F_q^*} \chi(udx_0+vdx_1+wd^2x_0x_1).
	\end{align*}
	It is not difficult to verify that $(x_0,x_1)\in C_{0,1}$ is such that $x_0x_1=0$ if and only if $x_1=0$ and $x_0=\pm 1$. Without loss of generality, we assume that $(1,0)\in C_{0,1}^{(+)}$. Let $\mathbf{1}_0:\F_q\rightarrow \{0,1\}$ be a function defined by $\mathbf{1}_0(0)=1$ and $\mathbf{1}_0(u)=0$ if $u\neq0$. We continue our calculation of $\sum_{c\in \sq}S(c)$.
	\begin{align}
	\nonumber	 &\sum_{x\in C_{0,1}^{(+)}}\sum_{d\in \F_q^*} \chi(udx_0+vdx_1+wd^2x_0x_1)\\
	\nonumber	=&\sum_{\substack{x\in C_{0,1}^{(+)}\\x_1\neq 0}}\sum_{d\in \F_q^*} \chi\left(wx_0x_1\left(d+\frac{ux_0+vx_1}{2wx_0x_1}\right)^2 - \frac{(ux_0 + vx_1)^2}{4wx_0x_1}\right)+ \sum_{d\in \F_q^*}\chi(ud)\\
	\nonumber	=&\sum_{\substack{x\in C_{0,1}^{(+)}\\x_1\neq 0}}\sum_{c\neq \frac{ux_0+vx_1}{2wx_0x_1}}\chi(wx_0x_1c^2) \chi\left(- \frac{(ux_0 + vx_1)^2}{4wx_0x_1}\right) + (\mathbf{1}_0(u)+1)\\
	\nonumber	=&\sum_{\substack{x\in C_{0,1}^{(+)}\\x_1\neq 0}}\left(\chi\left(- \frac{(ux_0 + vx_1)^2}{4wx_0x_1}\right)\sum_{c\in \F_q} \chi(wx_0x_1c^2)-\chi(0)\right)+ (\mathbf{1}_0(u)+1)\\
	\nonumber	
	=&\sum_{\substack{x\in C_{0,1}^{(+)}\\x_1\neq 0}}\left(\chi\left(- \frac{(ux_0 + vx_1)^2}{4wx_0x_1}\right)-1\right)+ (\mathbf{1}_0(u)+1), \qquad\text{(by Lemma \ref{le:chi(ac^2)})}.
	\end{align}
	Since $\#\left\{ x: x\in C_{0,1}^{(+)},x_1\neq 0 \right\}=\frac{q-1}{2}$, we obtain
	\begin{equation}\label{eq:chi_1_0} 	
		\sum_{c\in \sq}S(c)=\sum_{\substack{x\in C_{0,1}^{(+)}\\x_1\neq 0}}\chi\left(- \frac{(ux_0 + vx_1)^2}{4wx_0x_1}\right)+ \mathbf{1}_0(u) +  \frac{q+1}{2}.
	\end{equation}
	
	Similarly we can show that
	\begin{equation}\label{eq:chi_alpha_0}
		\sum_{c\in \sq}S(c\alpha)=\sum_{\substack{x\in C_{0,\alpha}^{(+)}\\x_0\neq 0}}\chi\left(- \frac{(ux_0 + vx_1)^2}{4wx_0x_1}\right)+ \mathbf{1}_0(v) +  \frac{q+1}{2}
	\end{equation}
	
	Now we turn to the character sums over $\mathbb{Q}(\zeta_p)$. Using \eqref{eq:C_0beta-1} and \eqref{eq:C_0beta-2}, we can verify that for $(x_0,x_1)$ and $(y_0,y_1)\in C_{0,1}$ with $x_1,y_1\neq0$, $\frac{x_0}{x_1}=\frac{y_0}{y_1}$ if and only if $(x_0,x_1)=(y_0,y_1)$ or $(x_0,x_1)=(-y_0,-y_1)$. Hence
	\begin{align}
	\nonumber	 &\sum_{\substack{x\in C_{0,1}^{(+)}\\x_1\neq 0}}\lambda\left(- \frac{(ux_0 + vx_1)^2}{4wx_0x_1}\right)\\
	\nonumber	=&\sum_{\substack{x\in C_{0,1}^{(+)}\\x_1\neq 0}}\lambda\left(- \frac{u^2}{4w}\frac{x_0}{x_1}- \frac{v^2}{4w}\frac{x_1}{x_0} \right)\lambda\left(-\frac{uv}{2w}\right)\\
	\nonumber	=&\frac{1}{2}\lambda\left(-\frac{uv}{2w}\right)\sum_{\substack{x\in C_{0,1}\\x_1\neq 0}}\lambda\left(- \frac{u^2}{4w}\frac{x_0}{x_1}- \frac{v^2}{4w}\frac{x_1}{x_0} \right)\\
	\label{eq:lambda_1}	=&\frac{1}{2}\sum_{\substack{x\in C_{0,1}\\x_1\neq 0}}\lambda\left(- \frac{u^2}{4w}\frac{x_0}{x_1}- \frac{v^2}{4w}\frac{x_1}{x_0} \right)\qquad \left(\Tr_{q/p}\left(\frac{uv}{w}\right)=0\right).
	\end{align}
	Similarly,
	\begin{equation}\label{eq:lambda_2}
		\sum_{\substack{x\in C_{0,\alpha}^{(+)}\\x_0\neq 0}}\lambda\left(- \frac{(ux_0 + vx_1)^2}{4wx_0x_1}\right)=\frac{1}{2}\sum_{\substack{x\in C_{0,\alpha}\\x_0\neq 0}}\lambda\left(- \frac{u^2}{4w}\frac{x_0}{x_1}- \frac{v^2}{4w}\frac{x_1}{x_0} \right).
	\end{equation}
	
	Let us first consider the case in which $q\equiv 1\pmod{4}$. In this case, $C_{0,\beta}$ is defined by \eqref{eq:C_0beta-1} and we can parameterize the points in $C_{0,1}$ and $C_{0,\alpha}$ as follows:
	\begin{eqnarray}
		C_{0,1} &=&\left\{\left(\frac{1-\alpha t^2}{1+\alpha t^2},\frac{2t}{1+\alpha t^2}\right): t\in\F_{q}^*\right\}\bigcup \{(\pm 1,0)\},\\
		C_{0,\alpha} &=&\left\{\left(\frac{2\alpha t}{\alpha+ t^2},\frac{\alpha-t^2}{\alpha+t^2}\right): t\in\F_{q}^*\right\}\bigcup \{(0,\pm 1)\}.
	\end{eqnarray}
	When $v=0$ and $u\neq0$, we continue the calculation of \eqref{eq:lambda_1} 
	\begin{align*}
		 &\frac{1}{2}\sum_{\substack{x\in C_{0,1}\\x_1\neq 0}}\lambda\left(- \frac{u^2}{4w}\frac{x_0}{x_1}- \frac{v^2}{4w}\frac{x_1}{x_0} \right)\\
		=&\frac{1}{2}\sum_{\substack{x\in C_{0,1}\\x_1\neq 0}}\lambda\left(- \frac{u^2}{4w}\frac{x_0}{x_1}\right)\\
		=&\frac{1}{2}\sum_{t\in \F_q^*}\lambda\left(- \frac{u^2}{4w}\frac{1-\alpha t^2}{2t}\right)\\
		=&\frac{1}{2}\sum_{t\in \F_q^*}\lambda\left( \frac{1}{t\left(-\frac{8w}{u^2}\right)}-\alpha\left(\frac{u^2}{8w}\right)^2  t\left(-\frac{8w}{u^2}\right)\right)\\
		=&\frac{1}{2} K\left(-\frac{u^4}{64w^2}\cdot \alpha \right).
	\end{align*}
	By \eqref{eq:chi_1_0}, Lemma \ref{le:liftingtoQ} and the above equation, we show that $\sum_{c\in \sq}S(c)\neq 0$ if and only if 
	\[K\left(-\frac{u^4}{64w^2}\cdot \alpha \right)\not\equiv 2 \pmod{4}.\]
	Similarly for $u=0$ and $v\neq0$, \eqref{eq:lambda_2} becomes
	\[\frac{1}{2}\sum_{\substack{x\in C_{0,\alpha}\\x_0\neq 0}}\lambda\left(- \frac{u^2}{4w}\frac{x_0}{x_1}- \frac{v^2}{4w}\frac{x_1}{x_0} \right)
	=\frac{1}{2} K\left(-\frac{v^4}{64w^2}\cdot \frac{1}{\alpha} \right).\]
	By \eqref{eq:chi_alpha_0}, Lemma \ref{le:liftingtoQ} and the above equation, we see that $\sum_{c\in \sq}S(c\alpha)\neq 0$ if and only if 
	\[K\left(-\frac{v^4}{64w^2}\cdot \frac{1}{\alpha}\right)\not\equiv 2 \pmod{4}.\]
	
	Next let consider the case in which $q\equiv 3\pmod{4}$. In this case, $C_{0,\beta}$ is defined by \eqref{eq:C_0beta-2}. Let $\tilde{\alpha} := \alpha - \theta_0^2$. We can parameterize the points in $C_{0,1}$ and $C_{0,\alpha}$ as follows:
	\begin{eqnarray}
		C_{0,1} &=&\left\{\left(\frac{1-2\theta_0 t- \tilde{\alpha} t^2}{1+\tilde{\alpha} t^2},\frac{2t}{1+\tilde{\alpha} t^2}\right): t\in\F_{q}^*\right\}\bigcup \{(\pm 1,0)\},\\
		C_{0,\alpha} &=&\left\{\left(\frac{2t/\alpha}{1+\tilde{\alpha} t^2},\frac{1-2\theta_0 t/\alpha^2- \tilde{\alpha} t^2}{1+\tilde{\alpha} t^2}\right): t\in\F_{q}^*\right\}\bigcup \{(0,\pm 1)\}.
	\end{eqnarray}
	When $v=0$ and $u\neq0$, \eqref{eq:lambda_1} becomes
	\begin{align*}
		&\frac{1}{2}\sum_{\substack{x\in C_{0,1}\\x_1\neq 0}}\lambda\left(- \frac{u^2}{4w}\frac{x_0}{x_1}- \frac{v^2}{4w}\frac{x_1}{x_0} \right)\\
		=&\frac{1}{2}\sum_{t\in \F_q^*}\lambda\left(- \frac{u^2}{4w}\frac{1-\tilde{\alpha} t^2}{2t}  + \frac{u^2}{4w}\theta_0\right)\\
		=&\frac{1}{2}\lambda\left(\frac{u^2\theta_0}{4w}\right)\sum_{t\in \F_q^*}\lambda\left( \frac{1}{t\left(-\frac{8w}{u^2}\right)}-\tilde{\alpha}\left(\frac{u^2}{8w}\right)^2  t\left(-\frac{8w}{u^2}\right)\right)\\
		=&\frac{1}{2}\lambda\left(\frac{u^2\theta_0}{4w}\right) K\left(-\frac{u^4}{64w^2}\cdot \tilde{\alpha} \right).
	\end{align*}
	From \eqref{eq:chi_1_0}, Lemma \ref{le:liftingtoQ} and the above equation, we deduce that $\sum_{c\in \sq}S(c)\neq 0$ if and only if 
	\[K\left(-\frac{u^4}{64w^2}\cdot \tilde{\alpha} \right)\not\equiv 0 \pmod{4}.\]
	Similarly for $u=0$ and $v\neq0$, \eqref{eq:lambda_2} becomes
	\[\frac{1}{2}\sum_{\substack{x\in C_{0,\alpha}\\x_0\neq 0}}\lambda\left(- \frac{u^2}{4w}\frac{x_0}{x_1}- \frac{v^2}{4w}\frac{x_1}{x_0} \right)
	=\frac{1}{2}\lambda\left(\frac{v^2\theta_0 \theta_0}{4w\alpha}\right) K\left(-\frac{v^4\alpha^2}{64w^2}\cdot \tilde{\alpha} \right).\]
	Again, from \eqref{eq:chi_alpha_0}, Lemma \ref{le:liftingtoQ} and the above equation, we see that $\sum_{c\in \sq}S(c\alpha)\neq 0$ if and only if 
	\[K\left(-\frac{v^4\alpha^2}{64w^2}\cdot \tilde{\alpha} \right)\not\equiv 0 \pmod{4}. \qedhere\]
\end{proof}
When $u$ and $v$ are both nonzero, from the proof of Theorem \ref{th:2rank-Kloostermansum}, we see that whether $\chi_{u,v,w}$ belongs $ \cK(\U_\theta)$ depends on the values of \eqref{eq:lambda_1} and \eqref{eq:lambda_2} modulo $4$, which are in general difficult to determine. Even for the case in Theorem \ref{th:2rank-Kloostermansum}, it is quite difficult to determine $K(a) \pmod{4}$ for an arbitrary $a\in \F_q$. However, for $p=\mathrm{char}(\F_q)=3$, there are several interesting results.
\begin{theorem} 
\cite{garaschuk_ternary_2008,gologlu_ternary_2012}\label{th:kloosterman_p=3}
	Let $a\in \F_{3^m}$. Then exactly one of the following cases occurs:
	\begin{enumerate}[label=(\alph*)]
		\item $a=0$ or $a$ is a square, $\Tr_{3^m/3}(\sqrt{a})\neq 0$ and $K(a)\equiv 1\pmod{2}$.
		\item $a=t^2-t^3$ for some $t\in \F_{3^m}\setminus\left\{ 0,1 \right\}$, at least one of $t$ and $1-t$ is a square and $K(a)\equiv 2m+2 \pmod{4}$. The number of $a\in \F_{3^m}^*$ such that $K(a)\equiv 2m+2 \pmod{4}$ is $\frac{5}{12}q-\frac{5}{4}$ if $m$ is odd, and $\frac{5}{12}q-\frac{3}{4}$ if $m$ is even.
		\item $a=t^2-t^3$ for some $t\in \F_{3^m}\setminus\left\{ 0,1 \right\}$, both $t$ and $1-t$ are nonsquares and $K(a)\equiv 2m \pmod{4}$. The number of $a\in \F_{3^m}^*$ such that $K(a)\equiv 2m \pmod{4}$ is $\frac{1}{4}(q+1)$ if $m$ is odd, and $\frac{1}{4}(q-1)$ if $m$ is even.
	\end{enumerate} 
\end{theorem}

By Theorems \ref{th:2rank-Kloostermansum} and \ref{th:kloosterman_p=3}, we can improve the lower bound for $\cK(\U_\theta)$ when $q$ is a power of $3$.
\begin{corollary}
	Let $q=3^m$ and $\U_\theta$ a unital embedded in $\Pi(x^2)$ defined in Lemma \ref{le:known_x^2}. Then
	\[\dim C_2(\U_\theta)\ge \left\{
	  \begin{array}{ll}
		\frac{2}{3}(q^3+q^2-2q)-1, & \hbox{$m$ is even;}\\[3pt]
	    \frac{2}{3}(q^3+q^2+q)-1, & \hbox{$m$ is odd.} 
	  \end{array}
	\right.
	\]	
\end{corollary}
\begin{proof}
	When $m$ is even, $-1$ is a square in $\F_q$. By Theorem \ref{th:kloosterman_p=3} and the fact that $\alpha$ is a nonsquare, $K(-\frac{u^4}{64w^2}\cdot \alpha) \not \equiv 2 \pmod{4}$ if and only if $K(-\frac{u^4}{64w^2}\cdot \alpha) \equiv 0 \pmod{4}$. A similar result can be obtained for $K(-\frac{v^4}{64w^2}\cdot \frac{1}{\alpha})$. The cardinalities of the following two sets 
	\[\left\{(u,w): w\neq0,K\left(-\frac{u^4}{64w^2}\cdot \alpha\right) 
	\equiv 0 \pmod{4}\right\},\]
	\[\left\{(v,w): w\neq0,K\left(-\frac{v^4}{64w^2}\cdot \frac{1}{\alpha}\right) \equiv 0 \pmod{4}\right\},\]
	both equal to 
	\[2(q-1)\# \{a\in \F_q: K(a)\equiv 0 \pmod{4}\}=\frac{1}{2}(q-1)^2.\]
	By Theorem \ref{th:2rank-Kloostermansum}, there are at least $2\cdot \frac{1}{2}(q-1)^2$ characters $\chi_{u,v,w}\in \cK(\U_\theta)$ when one of $u$ and $v$ equals $0$. Combining this result ($\Tr_{q/p}(uv/w^2)=0$) with the corresponding lower bound obtained in Theorem \ref{th:2rank_xiang} ($\Tr_{q/p}(uv\theta_1/w^2)\neq 0$ and $\theta_1=1$), we have
	\begin{align*}
		\#\cK(\U_\theta) &\ge (q^3-q^2+q)(1-\frac{1}{3}) + \frac{q^2}{3}+2\cdot \frac{1}{2}(q-1)^2\\
		&=\frac{2}{3}(q^3+q^2-2q)-1.
	\end{align*}
	
	When $m$ is odd, $\theta_0^2-\alpha$ is a nonsquare. By Theorem \ref{th:kloosterman_p=3}, $K\left(\frac{u^4}{64w^2}\cdot (\theta_0^2-\alpha) \right)\not\equiv 0 \pmod{4}$ if and only if $K\left(\frac{u^4}{64w^2}\cdot (\theta_0^2-\alpha) \right)\equiv 2 \pmod{4}$ and a similar result can be obtained for $K\left(\frac{v^4\alpha^2}{64w^2}\cdot (\theta_0^2-\alpha) \right)$. The cardinality of the following two sets 
	\[\left\{(u,w): w\neq0, K\left(\frac{u^4}{64w^2}\cdot (\theta_0^2-\alpha) \right)\equiv 2 \pmod{4}\right\},\]
	\[\left\{(v,w): w\neq0, K\left(\frac{v^4\alpha^2}{64w^2}\cdot (\theta_0^2-\alpha) \right) \equiv 2 \pmod{4}\right\},\]
	both equal to 
	\[2(q-1)\# \{a\in \F_q: K(a)\equiv 2 \pmod{4}\}=\frac{1}{2}(q^2-1).\]
	Again by Theorem \ref{th:2rank_xiang} and Theorem \ref{th:2rank-Kloostermansum}, we have
	\begin{align*}
		\#\cK(\U_\theta) &\ge (q^3-q^2+q)(1-\frac{1}{3}) + \frac{q^2}{3}+2\cdot \frac{1}{2}(q^2-1)\\
		&=\frac{2}{3}(q^3+q^2+q)-1.
	\end{align*}
	
	Together with \eqref{eq:dim_U=dim_U'} and \eqref{eq:dim=cK}, we get two lower bounds on $\dim C_2(\U_\theta)$.
\end{proof}

\section*{Acknowledgment}
The authors would like to thank Qing Xiang for pointing out the work of Junhua Wu \cite{wu_binarycodes_2012}. 
This work is supported by the Research Project of MIUR (Italian Office for University and Research) ``Strutture geometriche, Combinatoria e loro Applicazioni" 2012. Yue Zhou is partially supported by the National Natural Science Foundation of China (No.\ 11401579).


\end{document}